\documentclass{amsart}

\usepackage{amsfonts}
\usepackage{amssymb}
\usepackage{amsmath}
\usepackage{mathrsfs}

\newtheorem{theorem}{Theorem}[section]
\newtheorem{lemma}[theorem]{Lemma}
\newtheorem{corollary}[theorem]{Corollary}
\newtheorem{proposition}[theorem]{Proposition}

\theoremstyle{definition}
\newtheorem{definition}[theorem]{Definition}
\newtheorem{example}[theorem]{Example}

\newtheorem{remark}[theorem]{Remark}

\numberwithin{equation}{section}

\newcommand{\R}{\mathbb R}

\newcommand{\B}{\mathbb B}
\newcommand{\Hi}{\mathbb H}

\newcommand{\X}{\mathbb X}
\newcommand{\Y}{\mathbb Y}

\newcommand{\nullv}{\mathbf{0}}
\newcommand{\Sfer}{\mathbb S}
\newcommand{\Lbo}{\mathcal{L}}

\newcommand{\modconv}[1]{\delta_{#1}}
\newcommand{\lip}[2]{{\rm lip}\left(#1;#2\right)}

\newcommand{\ball}[2]{{\rm B}\left(#1,#2\right)}
\newcommand{\opic}[2]{{\rm o}\left(#1;#2\right)}
\newcommand{\dist}[2]{{\rm dist}\left(#1;#2\right)}

\newcommand{\dom}{{\rm dom}\, }

\newcommand{\cl}{{\rm cl}\, }
\newcommand{\bd}{{\rm bd}\, }
\newcommand{\inte}{{\rm int}\, }

\newcommand{\Fder}{{\rm D}}
\newcommand{\ncone}[2]{{\rm N}(#1,#2)}
\newcommand{\CII}{{\rm C}^{1,1}}

\newcommand{\Image}[2]{\mathcal{I}_{#1,#2}}
\newcommand{\Lagr}[2]{{\rm L}({#1;#2})}
\newcommand{\NlLagr}[2]{\mathscr{L}({#1;#2})}

\newcommand{\val}[1]{{\rm val}_{#1}}

\begin{document}

\title{On the Polyak convexity principle and its application to variational
analysis}

\author{A. Uderzo}

\address{Dipartimento di Matematica e Applicazioni, Universit\`a di Milano-Bicocca,
Via Cozzi, 53 - 20125 Milano, Italy}


\email{amos.uderzo@unimib.it}


\subjclass{Primary: 52A05; Secondary: 49J52, 90C46, 90C48.}

\date{\today}


\keywords{Uniformly convex Banach space, modulus of convexity,
$\CII$ mapping, metric regularity, constrained optimization, Lagrangian
duality, value function}

\begin{abstract}
According to a result due to B.T. Polyak,  a mapping between
Hilbert spaces, which is $\CII$ around a regular point, carries
a ball centered at that point to a convex set, provided that the
radius of the ball is small enough. 
The present paper considers the extension of such result to mappings
defined on a certain subclass of uniformly convex Banach spaces.
This enables one to extend to such setting a variational principle for
constrained optimization problems, already observed in finite
dimension, that establishes a convex behaviour for proper
localizations of them.
Further variational consequences are explored.
\end{abstract}

\maketitle


\section{Introduction}

The source of several deep results and intriguing problems in nonlinear
analysis can be found, to an attentive view, in the proficuous interplay
between smoothness and convexity. Sometimes, there is some smoothness
hidden in convexity. The generic (in fact, $G_\delta$ dense) G\^ateaux
differentiability of convex continuous functions defined on separable
Banach spaces, which was established by Mazur in 1933, paved the way
to a fruitful research line culminating with the theory of Asplund spaces
(see \cite{FaHaMoPeZi01,Mord06,Phel93}).
Symmetrically, some convexity is hidden in
smoothness. Indeed, smoothness at various levels provides powerful
and widely exploited criteria for detecting convexity of functions.
Another issue arosen within this interplay is how to recognize
convexity of images through smooth mappings. In fact, only a few classes
of mappings between vector spaces are known, of course besides
the linear ones, to guarantee convexity of images of convex subsets
of the domain space.
Yet, such a question seems to be of crucial interest in connection with
optimization and control theory related topics. For instance, the 
famous Lyapunov's convexity theorem on the range of a
nonatomic finite dimensional vector measure found a relevant
application in the formulation of the ``bang-bang" principle, a
fundamental result in control theory, as well as in several areas
of mathematical economics.

Historically, it seems not to be so easy to trace back an origin
for the problem of recognizing convexity of images of sets under
mappings. In this concern, one should not omit to mention the studies
on the convexity of images of spheres through vector quadratic forms,
which were triggered by the Toepliz-Hausdorff theorem (see
\cite{Barv95,Poly98} and references therein).
A significant step towards a theory embracing wide
classes of mappings between abstract spaces was made with the
appearance of a result due to B.T. Polyak (see \cite{Poly01,Poly03}).
He succeeded in proving that $\CII$ mappings between Hilbert
spaces, which are regular at a given point, carry balls centered
at that point to convex sets (with nonempty interior), provided the
radius of the balls are sufficiently small. The opinion maintained
by the author of the present paper is that, despite its importance,
such result (henceforth referred to as the Polyak convexity principle)
has not received so far an adequate attention, deserving
instead  a major popularization, especially among researchers
working in variational analysis and optimization areas. The aim
of the present paper is therefore to contribute to stimulate further
developments on this subject. This is done by showing that the
validity of the Polyak convexity principle is not limited to the
Hilbert space setting, but it can be extended to a certain subclass
of uniformly convex Banach spaces. These form a well-known subclass
of reflexive Banach spaces and are characterized by the rotund
shape of their balls, quantitatively described by their respective
moduli of convexity. A key element which makes possible the
extension to such setting of the aforementioned principle is a
condition on the asymptotic behaviour of the modulus of convexity,
to be combined with the smoothness assumption on the given mapping.
This is because, as already remarked by Polyak, the convexity
principle is not able to preserve convexity of images of general subsets,
but, relying on approximation/perturbation techniques of variational
analysis, it needs a certain ``rotund geometry" on the domain
space, which is able to guarantee a ``stable form" of convexity. In fact,
a possible way of looking at the Polyak principle is as at a solvability
result on smooth equations, with the known term subject to
perturbations restricted to a convex set.

A remarkable consequence of the Polyak convexity principle is that,
to a certain extent, $\CII$ smoothness accompained by regularity
yields a local convex behaviour of mappings. This fact may not be
so striking, if taking into account the nice characterization of $\CII$
functions in Hilbert spaces found by Hiriart-Urruty and Plazanet
in 1989 (see \cite{HirPla89}). According to it, a function $\phi:\Hi
\longrightarrow\R$ defined on a Hilbert space $(\Hi,\|\cdot\|)$ is
$\CII$ iff there is some positive constant $\alpha$ such that
$\phi+\alpha\|\cdot\|^2$ and $-\phi+\alpha\|\cdot\|^2$ are both
convex functions. In particular, $\CII$ functions are known to
be difference of convex functions (notice, again a manifestation of
the interplay between smoothness and convexity).

The benefic effects of the convexity hidden in smoothness can be
evidently appreciated when dealing with optimization problems.
On this theme, Polyak himself observed, on the base of the convexity
principle, that nonlinear problems in mathematical programming
with $\CII$ data behave like convex programs near regular feasible
points. It would be useful that such result, having notable consequences
both from the theoretical and computational point of view, could be
extended far beyond the finite dimensional setting, in which
has been first presented (see \cite{Poly01,Poly03}). An attempt to
proceed in this direction is made in the second part of the present
work. The convexity property of images of convex sets under smooth
mappings has been investigated also in \cite{BoEmKo04}, where
nonlocal sufficient conditions are proposed in a finite dimensional
setting.

The material exposed in the present paper is arranged in four
main sections, included the current one. Section \ref{sec:2} collects
miscellaneous notions from nonlinear analysis and geometrical
theory of Banach spaces. Related technical facts, which are needed in
the subsequent analysis, are established. The main result,
that is an extension of the Polyak convexity principle to an
adequate Banach space setting, is presented and discussed in
Section \ref{sec:3}. Section \ref{sec:4} is reserved to provide
applications of the main result to some topics of nonlinear
optimization in Banach spaces. In particular, a variational principle
on the convex behaviour of proper localizations of constrained
extremum problems with $\CII$ data is derived. Its consequences
on the Lagrangian duality and on problem calmness are
subsequently explored.


\section{Notations and preliminaries}     \label{sec:2}

Throughout the paper, whenever $(\X,\|\cdot\|)$ is a Banach space,
$\ball{x}{r}$ denotes the ball with centre at $x\in\X$ and radius
$r\ge 0$. The same notation is used also for balls in metric spaces.
The null vector of $\X$ is marked by $\nullv$. The
unit ball, i.e. the set $\ball{\nullv}{1}$, is simply denoted by
$\B$, whereas the unit sphere by $\Sfer$. Given $x_1\, x_2\in\X$,
the closed line segment with endpoints $x_1$ and $x_2$ is indicated
by $[x_1,x_2]$. If $S$ is a subset of a Banach space, $\inte S$,
$\bd S$ and $\cl S$ denote the interior, the boundary and the
(topological) closure of $S$, respectively.

\subsection{Uniformly convex Banach spaces and their moduli}

Given a Banach space $(\X,\|\cdot\|)$, some features of the geometry
of $\X$, related to the rotundity of its ball, can be quantitatively
described by means of the function $\modconv{\X}:[0,2]\longrightarrow
\R$, defined by
$$
    \modconv{\X}(\epsilon)=\inf\left\{1-\left\|\frac{x_1+x_2}{2}\right\|:\ 
    x_1,\, x_2\in\B,\ \|x_1-x_2\|\ge\epsilon\right\},
$$
which is called the {\em modulus of convexity} of $(\X,\|\cdot\|)$. It is possible
to prove that the modulus of convexity of a given Banach space admits
the following equivalent representations
\begin{eqnarray*}
     \modconv{\X}(\epsilon)& =& \inf\left\{1-\left\|\frac{x_1+x_2}{2}\right\|:\ 
    x_1,\, x_2\in\B,\ \|x_1-x_2\|=\epsilon\right\}    \\
     & = & \inf\left\{1-\left\|\frac{x_1+x_2}{2}\right\|:\ x_1,\, x_2\in\Sfer,\ \|x_1-x_2\|=
    \epsilon\right\}
\end{eqnarray*}
(see, for instance, \cite{FaHaMoPeZi01}).

\begin{definition}
A Banach space  $(\X,\|\cdot\|)$ is called {\em uniformly convex} if
it is $\modconv{\X}(\epsilon)>0$ for every $\epsilon\in (0,2]$.
\end{definition}

\begin{example}      \label{ex:lpLpWp}
All Hilbert spaces are uniformly convex. Indeed, by a straightforward
application of the parallelogram law it is possible to show that, if
$(\Hi,\|\cdot\|)$ is a Hilbert space, then it results in
$$
    \modconv{\Hi}(\epsilon)=1-\sqrt{1-\frac{\epsilon^2}{4}},\quad\forall
    \epsilon\in [0,2].
$$
The Banach spaces $l^p$, $L^p$, and $W^p_m$ are known to be uniformly
convex if $1<p<\infty$. In particular, if $p\ge 2$ their respective
moduli of convexity can be explicitly calculated. They turn out to be
$$
     \modconv{l^p}(\epsilon)=\modconv{L^p}(\epsilon)=\modconv{W^p_m}
    (\epsilon)=1-\left[1-\left(\frac{\epsilon}{2}\right)^p\right]^{1/p},\quad\forall
    \epsilon\in [0,2].
$$
If $1<p<2$, relying on the asymptotic behaviour of the modulus
of convexity, the following estimate from below is known to hold
$$
    \modconv{l^p}(\epsilon)=\modconv{L^p}(\epsilon)=\modconv{W^p_m}
    (\epsilon)> \frac{p-1}{8}\epsilon^2,\quad\forall\epsilon\in (0,2].
$$
\end{example}

\begin{example}
As a consequence of the James' characterization of weak compactness,
one can deduce that, if a Banach space is uniformly convex, then it
must be reflexive. By consequence, such spaces as $c_0$, $L^1$ and $L^\infty$
fail to be uniformly convex.
\end{example}

For the purposes of the present investigations, a geometrical
property of a special subclass of uniformly convex spaces is
needed. Loosely speaking, such property prescribes
a quadratic estimate from below for the distance
of the middle point of two elements in a ball from the boundary of
that ball. Not surprisingly, a sufficient condition for the validity of
such an estimate can be given in terms of modulus of convexity.

\begin{lemma}     \label{lem:uniconv}
Let $(\X,\|\cdot\|)$ be a uniformly convex Banach space. Suppose
that its modulus of convexity fulfils the condition
\begin{equation}    \label{in:modcocond}
    \modconv{\X}(\epsilon)\ge c\epsilon^2,\quad\forall\epsilon\in
    [0,2],
\end{equation}
for some $c>0$. Then, for every $x_0,\, x_1,\, x_2\in\X$ and
$r>0$, with $x_1,\, x_2\in\ball{x_0}{r}$, it holds
$$
    \ball{\frac{x_1+x_2}{2}}{\frac{c\|x_1-x_2\|^2}{r}}\subseteq
    \ball{x_0}{r}.
$$
\end{lemma}

\begin{proof}
Fix $r>0$.
Since the distance induced on $\X$ by $\|\cdot\|$ is invariant
under translations, without loss of generality it is possible to
assume that $x_0=\nullv$. By using one of the possible
representations of $\modconv{\X}$, one has
$$
    \inf_{x_1,x_2\in r\B\atop \|x_1-x_2\|= r\epsilon}\left\{r
       -\left\|\frac{x_1+x_2}{2}\right\| \right\}= r\cdot
     \inf_{u_1,u_2\in \B\atop \|u_1-u_2\|=\epsilon}\left\{1-\left\|
    \frac{u_1+u_2}{2}\right\|\right\}=r\modconv{\X}(\epsilon),
     \quad\forall \epsilon\in [0,2].
$$
By virtue of condition $(\ref{in:modcocond})$ one obtains
$$
     r\ge\sup_{x_1,x_2\in r\B\atop \|x_1-x_2\|=r\epsilon}
       \left\|\frac{x_1+x_2}{2}\right\|+rc\epsilon^2,
     \quad\forall \epsilon\in [0,2].
$$
This amounts to say that
$$
    \left\|\frac{x_1+x_2}{2}\right\|+\frac{c\|x_1-x_2\|^2}{r}\le r,
  \quad \forall x_1,\, x_2\in r\B,\ \|x_1-x_2\|=r\epsilon.
$$
Since the lasty inequality is true for every $\epsilon\in[0,2]$,
it follows
$$
    \left\|\frac{x_1+x_2}{2}\right\|+\frac{c\|x_1-x_2\|^2}{r}\le r,
  \quad \forall x_1,\, x_2\in r\B.
$$
Thus, by applying the triangle inequality, whenever $\hat x\in
\ball{\frac{x_1+x_2}{2}}{\frac{c\|x_1-x_2\|^2}{r}}$ one
obtains
$$
     \|\hat x\|\le\left\|\frac{x_1+x_2}{2}\right\|+
     \left\|\frac{x_1+x_2}{2}-\hat x\right\|\le
  \left\|\frac{x_1+x_2}{2}\right\|+\frac{c\|x_1-x_2\|^2}{r}\le r,
$$
which completes the proof.
\end{proof}

\begin{remark}    \label{rem:convmodHilb}
Notice that, whenever $(\X,\|\cdot\|)$ is in particular a Hilbert
space, in the light of what has been noted in Example \ref {ex:lpLpWp},
condition $(\ref{in:modcocond})$ turns out to be satisfied with
$c=1/8$. Besides, all spaces $l^p$, $L^p$, and $W^p_m$, with $1<p<2$,
admits a modulus of convexity satisfying $(\ref{in:modcocond})$
with $c=\frac{p-1}{8}$.
\end{remark}

For further details on the theory of uniformly convex Banach
spaces and their moduli of convexity the reader is referred to
\cite{Dies85,FaHaMoPeZi01,Milm71}.


\subsection{Some properties of $\CII$ mappings}

Let $(\X,\|\cdot\|)$ and $(\Y,\|\cdot\|)$ be Banach spaces. The
Banach space of all bounded linear operators bewteen $\X$
and $\Y$, equipped with the operator norm, is denoted by $(\Lbo
(\X,\Y),\|\cdot\|_\Lbo)$. The space $\Lbo(\X,\R)$ is simply marked
by $\X^*$, with $\langle\cdot,\cdot\rangle:\X^*\times\X\longrightarrow
\R$ denoting the duality pairing $\X^*$ with $\X$.
The null vector of a dual space is marked by $\nullv^*$.
If $S\subseteq\X$ is a nonempty set, $S^\ominus=\{x^*\in\X^*:\ 
\langle x^*,x\rangle\le 0,\quad\forall x\in S\}$ represents the
negative dual cone of $S$, while $S^\perp=S^\ominus\cap
(-S^\ominus)$ the annihilator of $S$. If $x_0\in S$, $\ncone{x_0}
{S}$ stands for the normal cone of $S$ at $x_0$ in the sense of
convex analysis.
Given a mapping $f:\Omega\longrightarrow\Y$, with $\Omega$ open
subset of $\X$, and given $x_0\in\Omega$, the Fr\'echet derivative
of $f$ at $x_0$ is denoted by $\Fder f(x_0)\in\Lbo(\X,\Y)$.
If $f$ is Fr\'echet differentiable at $x_0$, the remainder
in its first-order expansion is denoted by $\opic{x_0}{\cdot}$,
i.e.
$$
   \opic{x_0}{h}=f(x_0+h)-f(x_0)-\Fder f(x_0)[h],\quad h\in\X,
   \ x_0+h\in\Omega.
$$
If a mapping $f:\Omega\longrightarrow\Y$ is Fr\'echet
differentiable at each point of $\Omega$ and the mapping
$\Fder f:\Omega\longrightarrow\Lbo(\X,\Y)$ is Lipschitz continuous
on $\Omega$, $f$ is said to be $\CII$ on $\Omega$. The space of
all such mappings is indicated by $\CII(\Omega)$. If $f\in\CII
(\Omega)$, the infimum over all values $\kappa>0$ such that
$$
    \|\Fder f(x_1)-\Fder f(x_2)\|_\mathcal{L}\le\kappa\|x_1-x_2\|,
   \quad\forall  x_1,\, x_2\in\Omega,
$$
is called modulus of Lipschitz continuity of $\Fder f$ on $\Omega$
and is indicated by $\lip{\Fder f}{\Omega}$.

The proof of a lemma useful in the sequel involves elements of
integral calculus for mappings between Banach spaces. In this
concern, take into account that, given a compact interval $[a,b]\subseteq\R$ and
$f:[a,b]\longrightarrow\Y$, its integral over $[a,b]$, denoted by
$\int_a^bf(t)\, {\rm d}t$, is to be intended in the sense of Gavurin.
Roughly speaking, this means that
such integral can be defined by partitioning $[a,b]$
into finitely many subintervals and by taking the limit of the
integral sum as the partition mesh width goes to $0$. It has
been shown that every continuous mapping is integrable in this
sense. As a further step, given a mapping $G:\X\longrightarrow
\Lbo(\X,\Y)$ and $x_0, \, h\in\X$, define
$$
    \int_{x_0}^{x_0+h}G(x)\, {\rm d}x=\int_0^1G(x_0+th)[h]\, {\rm d}t.
$$
In such setting, an analogous of the fundamental theorem of classical
integral calculus can be stated as follows.

\begin{theorem}    \label{thm:fundcalint}
Let $f:\X\longrightarrow\Y$ be a mapping between Banach spaces, let
$\Omega$ be an open subset of $\X$, and let $x_0\in\Omega$, $h\in\X$
with $[x_0,x_0+h]\subseteq\Omega$. If $f\in {\rm C}^1(\Omega)$, then
$\int_{x_0}^{x_0+h}\Fder f(x)\, {\rm d}x$ exists and
$$
    \int_{x_0}^{x_0+h}\Fder f(x)\, {\rm d}x=f(x_0+h)-f(x_0).
$$
\end{theorem}

\noindent For more details on this topic see \cite{KanAki82} and
references therein. One is now in a position to establish a lemma,
which will be used in the subsequent section.

\begin{lemma}     \label{lem:quadestim}
Let $f:\X\longrightarrow\Y$ be a mapping between Banach spaces, let $\Omega$
be an open subset of $\X$, and let $x_1,\, x_2\in\Omega$, with $[x_1,x_2]\subseteq
\Omega$. If $f\in\CII(\Omega)$ and $\bar x=\frac{x_1+x_2}{2}$, then it holds
\begin{equation*}
  \|\opic{\bar x}{x_1-x_2}\|\le\frac{\lip{\Fder f}{\Omega}}{8}\|x_1-x_2\|^2.
\end{equation*}
Consequently,
\begin{equation}    \label{in:2ndordest}
    \left\|\frac{f(x_1)+f(x_2)}{2}-f\left(\frac{x_1+x_2}{2}\right) \right\|\le
    \frac{\lip{\Fder f}{\Omega}}{16}\|x_1-x_2\|^2.
\end{equation}
\end{lemma}

\begin{proof}
The first assertion is a well-known result, whose proof
follows a standard argument based on Theorem \ref{thm:fundcalint}
and is provided here for the sake of completeness.
From the first-order expansion of function $f$ near $\bar x$,
one obtains
\begin{eqnarray*}
    \opic{\bar x}{x_1-x_2} &=& \left\|f(x_1)-f(\bar x)-\Fder f(\bar x)
     \left[\frac{x_1-x_2}{2}\right]\right \|= \left\|\int_0^1\left(\Fder
    f(\bar x+t(x_1-\bar x))-\Fder f(\bar x)\right) 
    \left[\frac{x_1-x_2}{2}\right]{\rm d}t \right\|  \\
    & \le &\int_0^1\|\Fder f(\bar x+t(x_1-\bar x))-\Fder f(\bar x) \|_\mathcal{L}
      \left\|\frac{x_1-x_2}{2}\right\|{\rm d}t \le
    \frac{\lip{\Fder f}{\Omega}}{4}\|x_1-x_2\|^2\int_0^1t{\rm d}t \\
    &= & \frac{\lip{\Fder f}{\Omega}}{8}\|x_1-x_2\|^2.
\end{eqnarray*}
As for the second assertion, by adding up the two below first-order
expansions of mapping $f$ at $\bar x$
$$
    f(x_i)=f(\bar x)+\Fder f(\bar x)[x_i-\bar x]+
    \opic{\bar x}{x_i-\bar x},\quad i=1,\, 2,
$$
and dividing by $2$, one gets
$$
   \frac{f(x_1)+f(x_2)}{2}=f(\bar x)+\frac{1}{2}
   \left\{ \Fder f(\bar x)\left[\frac{x_1-x_2}{2}\right]+
    \Fder f(\bar x)\left[\frac{x_2-x_1}{2}\right]\right\}+
   \opic{\bar x}{\frac{x_1-x_2}{2}}+\opic{\bar x}{\frac{x_2-x_1}{2}}.
$$
Thus, by linearity of the derivative $\Fder f(\bar x)$, one obtains
$$
     \left\|\frac{f(x_1)+f(x_2)}{2}-f(\bar x)\right\|\le
    \left\|\opic{\bar x}{\frac{x_1-x_2}{2}}\right\|+
    \left\|\opic{\bar x}{\frac{x_2-x_1}{2}}\right\|.
$$
The inequality to be proved can be easily derived from the last one
by taking into account the estimate provided in the first part of
the thesis.
\end{proof}


\subsection{Metric regularity and linear openness}

A key assumption playing a crucial role in the proof of the main
result is local metric regularity. Recall that a mapping $f:\X\longrightarrow
\Y$ between Banach spaces is said to be {\em metrically regular}
around $(x_0,f(x_0))$, with $x_0\in\X$, if there exist positive
$\delta$, $\zeta$ and $\mu$ such that
$$
   \dist{x}{f^{-1}(y)}\le\mu\|y-f(x)\|,\quad\forall x\in\ball{x_0}{\delta},
   \ \forall y\in\ball{f(x_0)}{\zeta},
$$
where $\dist{x}{S}=\inf_{s\in S}\|s-x\|$ denotes the distance of $x$
from set $S$. For mappings which are strictly differentiable
at $x_0$ (and hence, for mappings $\CII$ in an open neighbourhood
of $x_0$) the following celebrated criterion for metric regularity
holds (see, for instance, Theorem 1.57 in \cite{Mord06})

\begin{theorem}[{\bf Lyusternik-Graves}]
Let $f:\X\longrightarrow\Y$ be a mapping between Banach spaces.
Suppose $f$ to be strictly differentiable at $x_0\in\X$. Then
$f$ is metrically regular around $(x_0,f(x_0))$ iff
$\Fder f(x_0)\in\Lbo(\X,\Y)$ is onto.
\end{theorem}

An equivalent reformulation of metric regularity will be also
exploited in the sequel, which refers to a local surjection
property known as {\em openness at a linear rate} around $(x_0,
f(x_0))$. It postulates the existence of positive $\delta$, $\zeta$
and $\sigma$ such that
\begin{equation}      \label{in:loclinopen}
    f(\ball{x}{r})\supseteq\ball{f(x)}{\sigma r}\cap\ball{f(x_0)}{\zeta},
      \quad\forall x\in\ball{x_0}{\delta},\ \forall r\in [0,\delta).
\end{equation}
Actually, metric regularity and openness at a linear rate describe
a Lipschitzian behaviour of mappings, which can be considered
in the more general setting of metric spaces. Given a mapping
$f:X\longrightarrow Y$ between metric spaces, for the purposes
of the present analysis it is convenient to recall also the notion
of openness with respect to a given subset $S\subset X$. Mapping
$f$ is said to be {\em open at a linear rate} on $S$ if there exists
$\sigma>0$ such that for every $x\in S$ and every $r>0$, with
$\ball{x}{r}\subseteq S$, it holds
$$
   f(\ball{x}{r})\supseteq\ball{f(x)}{\sigma r}.
$$
A relevant consequence that openness on a given subset
bears, already if considered in metric spaces, is stated in
the next lemma, whose proof can be found for instance in
\cite{DmMiOs80}.

\begin{lemma}          \label{lem:complopen}
Let $f:X\longrightarrow Y$ be a mapping between metric spaces
and let $S\subseteq X$. Suppose that:

\noindent $(i)$ $X$ is metrically complete, whereas the metric of
$Y$ is invariant under translation;

\noindent $(ii)$ $f\in {\rm C}(X)$ and is open at a linear rate on $S$;

\noindent $(iii)$ $\inte S\ne\varnothing$;

\noindent Then, also $Y$ is metrically complete.
\end{lemma}

The above lemma is employed next to prove the following property
useful in the sequel.

\begin{lemma}       \label{lem:closeimage}
Let $f:\X\longrightarrow\Y$ be a mapping between Banach spaces,
let $\Omega\subseteq\X$ be an open set and let $x_0\in\Omega$.
Suppose that $f\in {\rm C}(\Omega)$ and it is open at a linear rate around
$(x_0,f(x_0))$. Then, there exists $r_0>0$ such that, for every 
$r\in (0,r_0]$ the set $f(\ball{x_0}{r})$ is closed.
\end{lemma}

\begin{proof}
Since $\Omega$ is open, it is possible to take $\tilde r>0$ in such a
way that $\ball{x_0}{\tilde r}\subseteq\Omega$. Notice that, as $f$
is continuous at $x_0$, inclusion $(\ref{in:loclinopen})$, valid owing
to the linear openness of $f$, entails the existence of $\tilde\delta>0$
such that
\begin{equation}      \label{in:linopvar}
     f(\ball{x}{r})\supseteq\ball{f(x)}{\sigma r}, \quad\forall
   x\in\ball{x_0}{\tilde\delta},\ \forall r\in [0,\tilde\delta).
\end{equation}
Indeed, corresponding to $\zeta$ one can find $\delta_\zeta>0$
such that
$$
   f(\ball{x_0}{\delta_\zeta})\subseteq\ball{f(x_0)}{\zeta/2}.
$$
Thus, by taking $\tilde\delta>0$ in such a way that
$$
    \tilde\delta<\min\left\{\tilde r,\, \delta_\zeta,\, \frac{\zeta}
   {2\sigma},\, \delta\right\},
$$
one obtains that, whenever $x\in\ball{x_0}{\tilde\delta}$, it is
$f(x)\in\ball{x_0}{\zeta/2}$. It follows that, if $t\in (0,\tilde\delta]$,
and hence $\sigma t<\zeta/2$, one has
$$
    \ball{f(x)}{\sigma t}\subseteq\ball{f(x_0)}{\zeta}.
$$
Being $\tilde\delta<\delta$, the last inclusion reduces
$(\ref{in:loclinopen})$ to $(\ref{in:linopvar})$. Set $r_0=\tilde
\delta$ and fix an arbitrary $r\in (0,r_0]$. Put in that form,
openness at a linear rate around $(x_0,f(x_0))$ implies openness
on $\ball{x_0}{r/2}$, because only balls satisfying $\ball{x}{t}
\subseteq\ball{x_0}{r/2}$ must be considered. One is then in a
position to apply Lemma \ref{lem:complopen},
with $X=\ball{x_0}{r}$, $S=\ball{x_0}
{r/2}$ and $Y=f(\ball{x_0}{r})$. It follows that $f(\ball{x_0}{r})$
is metrically complete and, as such, it must be a closed subset of
$\Y$. This completes the proof.
\end{proof}


\hskip1cm

\section{The Polyak convexity principle in Banach spaces}    \label{sec:3}

Before entering the main result of the paper, to make easier the
presentation of its proof, a fact concerning convexity of sets is
explicitly stated, whose proof can be obtained without difficulty.

\begin{lemma}      \label{lem:convmidpoint}
Let $S\subseteq\Y$ be a closed subset of a Banach space. $S$ is
convex iff $\frac{1}{2}(y_1+y_2)\in S$, whenever $y_1,\ y_2\in S$.
\end{lemma}

\begin{theorem}    \label{thm:Polyakteo}
Let $f:\X\longrightarrow\Y$ be a mapping between Banach spaces, let $\Omega$
be an open subset of $\X$, let $x_0\in\Omega$, and $r>0$ such that $\ball{x_0}{r}
\subseteq\Omega$. Suppose that:

\noindent $(i)$ $(\X,\|\cdot\|)$ is uniformly convex and its
convexity modulus fulfils condition $(\ref{in:modcocond})$;

\noindent $(ii)$ $f\in\CII(\Omega)$ and $\Fder f(x_0)\in\Lbo(\X,\Y)$
is onto.

\noindent Then, there exists $\epsilon_0\in (0,r)$ such that
$f(\ball{x_0}{\epsilon})$ is convex, for every $\epsilon\in
[0,\epsilon_0]$.
\end{theorem}

\begin{proof}
Under hypothesis $(ii)$ it is possible to invoke the Lyusternik-Graves
theorem. According to it, mapping $f$ is locally metrically regular
around $(x_0,f(x_0))$. This means that there exist $\mu>0$,
$\zeta>0$, and $\delta_\mu>0$ such that
\begin{equation}     \label{in:metricreg}
    \dist{x}{f^{-1}(y)}\le\mu\|y-f(x)\|,\quad\forall x\in\ball{x_0}{\delta_\mu},
   \ \forall y\in\ball{f(x_0)}{\zeta}.
\end{equation}
By continuity of $f$ at $x_0$, corresponding to $\zeta$ there
exists $\delta_\zeta>0$ such that
$$
    f(x)\in\ball{f(x_0)}{\zeta},\quad\forall x\in\ball{x_0}{\delta_\zeta}.
$$
Since $f$ is continuous on $\Omega$ and open at a linear rate around
$(x_0,f(x_0))$, by virtue of Lemma \ref{lem:closeimage} there exists
$r_0>0$ such that $f(\ball{x_0}{t})$ is closed for every $t\in (0,r_0]$.
Now, take $\epsilon_0$ in such a way that
$$
   0<\epsilon_0<\min\left\{r,r_0,\delta_\mu,\delta_\zeta,
  \frac{8c}{\mu(\lip{\Fder f}{\Omega}+1)}\right\},
$$
where $c$ is as in $(\ref{in:modcocond})$, and fix an
arbitrary $\epsilon\in (0,\epsilon_0]$, the case $\epsilon=0$ being
trivial. In the light of Lemma \ref{lem:convmidpoint}, in order to
show that $f(\ball{x_0}{\epsilon})$ is convex, it suffices to prove
that, taken any pair $y_1,\, y_2\in f(\ball{x_0}{\epsilon})$ and
set
$$
    \bar y=\frac{y_1+y_2}{2},
$$
then also $\bar y$ happens to belong to $f(\ball{x_0}{\epsilon})$.
To this aim, corresponding to $y_1,\, y_2$, take $x_1,\, x_2\in
\ball{x_0}{\epsilon}$ such that $f(x_1)=y_1$ and $f(x_2)=y_2$
and define
$$
    \bar x=\frac{x_1+x_2}{2}.
$$
Since it is $\epsilon<\delta_\zeta$, the continuity of $f$ at $x_0$
implies
$$
   y_1,\, y_2\in\ball{f(x_0)}{\zeta},
$$
and hence $\bar y\in\ball{f(x_0)}{\zeta}$. Thus, since it is also
$\epsilon<\delta_\mu$,  then, being $(\bar x,\bar y)\in\ball{x_0}{\delta_\mu}
\times\ball{f(x_0)}{\zeta}$, by recalling inequality (\ref{in:metricreg})
one obtains
\begin{equation}    \label{in:metricregbar}
    \dist{\bar x}{f^{-1}(\bar y)}\le\mu\|\bar y-f(\bar x)\|.
\end{equation}
If it is $\bar y=f(\bar x)$, one achieves immediately what
was to be proved. So, suppose that $\|\bar y-f(\bar x)\|>0$. From
inequality  (\ref{in:metricregbar}) it follows that, corresponding
to $2\mu$, there exists $\hat x\in f^{-1}(\bar y)$, such that
$$
    \|\hat x-\bar x\|<2\mu\|\bar y-f(\bar x)\|.
$$
In force of hypothesis $(i)$ it is possible to apply the estimate
$(\ref{in:2ndordest})$ in Lemma \ref{lem:quadestim},
according to which one finds
$$
    \|\hat x-\bar x\|\le 2\mu \frac{\lip{\Fder f}{\Omega}}{16}\|x_1-x_2\|^2.
$$
Therefore, since by the above positions it is
$$
    \frac{\mu(\lip{\Fder f}{\Omega}+1)}{8c}<\frac{1}{\epsilon},
$$
it results in
$$
    \hat x\in\ball{\bar x}{\frac{c\|x_1-x_2\|^2}{\epsilon}}.
$$
According to Lemma \ref{lem:uniconv}, this fact is known to imply
that $\hat x\in\ball{x_0}{\epsilon}$, by virtue of the condition
$(\ref{in:modcocond})$ assumed on the convexity modulus of $(\X,\|\cdot\|)$.
Thus, $\bar y$ has been proved to belong to $f(\ball{x_0}{\epsilon})$,
so the proof is complete.
\end{proof}

\begin{remark}     \label{rem:Polyakteo}
(i) Since, as noticed in Remark \ref{rem:convmodHilb}, every Hilbert
space is an uniformly convex Banach space, whose modulus of convexity
fulfils condition $(\ref{in:modcocond})$, Theorem \ref{thm:Polyakteo}
is actually an extension of the Polyak convexity principle. Notice
that no assumption on the geometry of the range space $\Y$ has
been made.

(ii) The regularity condition requiring $\Fder f(x_0)$ to be onto can
not be dropped out, even in the case of very simple mappings acting
in finite-dimensional spaces. Consider, indeed, $f:\R^2\longrightarrow
\R^2$ defined by
$$
   f(x_1,x_2)=((x_1+x_2),(x_1+x_2)^2),
$$
and $x_0=(0,0)=\nullv$, $\R^2$ being equipped with its usual Hilbert
space structure. Mapping $f\in {\rm C}^2(\R^2)$, so it belongs
to $\CII(\inte\ball{\nullv}{r})$, for a proper $r>0$. Its Jacobian
matrix has rank $1$ at $\nullv$, so $\Fder f(\nullv)$ can not cover
$\R^2$. It is readily seen that, for every $\epsilon>0$, it results
in
$$
    f(\ball{\nullv}{\epsilon})=\{(y_1,y_2)\in\R^2: y_2=y_1^2,\, y_1\in
    [-\sqrt{2}\epsilon,\sqrt{2}\epsilon]\},
$$
which is not a convex subset of $\R^2$. Since, as a mapping
defined on the Hilbert space $\R^2$, $f$ satisfies all hypotheses
of Theorem \ref{thm:Polyakteo}, this example shows also that a
mapping carrying small balls to convex sets may happen to carry
convex subsets of such balls to nonconvex sets.

(iii) The next example shows that one can not hope to extend Theorem
\ref{thm:Polyakteo} out of the class of uniformly convex Banach spaces.
Suppose $\R^2$ to be equipped with the norm $\|x\|_\infty=\max
\{|x_1|,|x_2|\}$, which makes $\R^2$ not uniformly convex. Consider
the mapping $f:\R^2\longrightarrow\R^2$ defined by
$$
   f(x_1,x_2)=(x_1,x_1^2+x_2),
$$
and $x_0=(0,0)=\nullv$. Since $f\in {\rm C}^2(\R^2)$, it is also
$\CII(\inte\ball{\nullv}{r})$, for a proper $r>0$.  Moreover, being
$\Fder f(x)$ represented by the matrix
$$
   \left( \begin{array}{cc}
                    1 & 0 \\
                     2x_1 & 1 
         \end{array} \right),
$$
the linear mapping $\Fder f(x)$ is onto for every $x\in\R^2$.
Nonetheless, since now $\ball{\nullv}{\epsilon}=[-\epsilon,\epsilon]
\times[-\epsilon,\epsilon]$, it results in
$$
   f(\ball{\nullv}{\epsilon})=\bigcup_{t\in [-\epsilon,\epsilon]}
   \{(y_1,y_2)\in\R^2: y_2=y_1^2+t,\, y_1\in  [-\epsilon,\epsilon]\},
$$
which can be convex only if $\epsilon=0$.

(iv) The following complement of Theorem \ref{thm:Polyakteo},
already remarked in \cite{Poly01}, is worth being mentioned.
From hypothesis (ii) one has that $f(\inte \ball{x_0}{\epsilon})
\subseteq\inte f(\ball{x_0}{\epsilon})\ne\varnothing$, for every
$\epsilon\in(0,\epsilon_0]$. Therefore, it holds
$$
    f^{-1}(\bd f(\ball{x_0}{\epsilon}))\subseteq
    \bd\ball{x_0}{\epsilon}.
$$
\end{remark}

An interesting question related to Theorem \ref{thm:Polyakteo}
is whether it can be extended to some classes of nonsmooth mappings.
In consideration of the importance of nonsmooth analysis in
optimization, this further development would be remarkable and
widely motivated.
Reduced to its basic elements, as a  matter of fact, the proof
of Theorem \ref{thm:Polyakteo} consists in a proper combination of
distance estimates relying on a rotund geometry and metric regularity.
The latter has been well understood also for nonsmooth mappings and
adequately characterized in terms of generalized derivatives 
(see, for a thorough account on the subject, \cite{Mord06}).
Nonetheless, within the current approach, a developement in this direction
seems to be hardly possible. In this regard, a counterexample has
been devise by A.D. Ioffe showing that already ${\rm C}^1$ mappings
may happen to do not satisfy the thesis of the Polyak convexity
principle. Apart from the need of Lipschitz continuity of the derivative
mapping, another reason of difficulty is the role crucially played
by linearity in the estimates provided by Lemma \ref{lem:quadestim}
as well as in preserving convexity of sets. In both such circumstances
a successful replacement of linear mappings with merely positively
homogeneous first order approximations seems to be hardly
practicable.


\section{Applications to optimization}      \label{sec:4}

\subsection{A variational principle on the convex behaviour
of extremum problems}

Carrying on a reasearch line proposed in \cite{Poly01,Poly03},
this subsection is concerned with the study of local aspects
of the theory of constrained optimization problems of the
following form
$$
    \min_{x\in\X} \varphi(x)\quad\hbox{subject to}\quad g(x)\in C,
    \leqno (\mathcal{P})
$$
where the cost functional $\varphi:\X\longrightarrow\R\cup\{\pm
\infty\}$, the constraining mapping $g:\X\longrightarrow\Y$ and set
$C\subseteq\Y$ are given problem data. The feasible region of
$(\mathcal{P})$ is denoted here by $R=\{x\in\X:\ g(x)\in C\}=g^{-1}(C)$.
Set $\mathcal{Q}=(-\infty,0)\times C\subseteq\R\times\Y$ and fix
$x_0\in\X$. Following a wide-spread approach in optimization
(see, among others, \cite{DmMiOs80}),
the analysis of various features of
$(\mathcal{P})$ can be performed by associating with
that problem and with an element $x_0$ a mapping $\Image{\mathcal{P}}
{x_0}:\X\longrightarrow\R\times\Y$, defined as
\begin{equation}     \label{eq:imagemap}
    \Image{\mathcal{P}}{x_0}(x)=(\varphi(x)-\varphi(x_0),g(x)).
\end{equation}
Such mapping allows one to characterize the optimality of $x_0$,
as stated in the below remark.

\begin{remark}
An element $x_0\in R$ is a local solution to $(\mathcal{P})$ iff
there exists $r>0$ such that
$$
   \Image{\mathcal{P}}{x_0}(\ball{x_0}{r})\cap\mathcal{Q}
   =\varnothing. 
$$
Letting $\ball{x_0}{+\infty}=\X$, the above disjunction with $r=+\infty$
obviously characterizes global optimality of $x_0$.
\end{remark}

Such characterization is applied  in the next result to establish
a variational principle involving the classical Lagrangian function
${\rm L}:\Y^*\times\X\longrightarrow\R\cup\{\pm\infty\}$
$$
   \Lagr{y^*}{x}=\varphi(x)+\langle y^*,g(x)\rangle.
$$
Given $\epsilon>0$ and $x_0\in\R$, by a {\em $\epsilon$-localization}
of problem $(\mathcal{P})$ around $x_0$, the following extremum
problem is meant
$$
    \min_{x\in\ball{x_0}{\epsilon}}\varphi(x)\quad\hbox{subject to}
    \quad g(x)\in C.   \leqno (\mathcal{P}_{x_0,\epsilon})
$$
Notice that $(\mathcal{P}_{x_0,\epsilon})$ has the same objective
function as $(\mathcal{P})$, but its feasible region results from
$\ball{x_0}{\epsilon}\cap R$.

An element $x_0\in R$ is said to be regular for $(\mathcal{P})$
if $\varphi,\, g\in\CII(\Omega)$, where $\Omega$ is an open set
containing $x_0$, and mapping $\Image{\mathcal{P}}{x_0}$ is regular
at $x_0$ in the classical sense, i.e. mapping $\Fder (\varphi,g)
(x_0)$ is onto.

The variational principle, which is going to be presented next,
states that, in an adequate setting, around each regular point for
$(\mathcal{P})$ and
corresponding to each $\epsilon$ small enough, there exists
a $\epsilon$-localization of $(\mathcal{P})$ admitting a solution,
which further minimizes $\Lagr{y^*}{\cdot}$, for a proper $y^*\in
\Y^*$.

\begin{theorem}     \label{pro:Lagroptcond}
With reference to problem $(\mathcal{P})$, let $\Omega\subseteq\X$
be an open set and let $x_0\in R\cap\Omega$. Suppose that

\noindent $(i)$ $(\X,\|\cdot\|)$ is uniformly convex and its
convexity modulus fulfils condition $(\ref{in:modcocond})$;

\noindent $(ii)$ $(\Y,\|\cdot\|)$ is a reflexive Banach space; 

\noindent $(iii)$ $\varphi,\, g\in\CII(\Omega)$ and $\Fder(\varphi,g)
(x_0)\in\Lbo(\X,\R\times\Y)$ is onto.

\noindent Then, there exists a positive $\epsilon_0$ such that for every
$\epsilon\in (0,\epsilon_0]$ there are $x_\epsilon\in\bd\ball{x_0}
{\epsilon}$ and $\lambda_\epsilon\in\Y^*$ with the properties:
\begin{equation}     \label{sol:optim}
   \hbox{$x_\epsilon$ solves problem $(\mathcal{P}_{x_0,\epsilon})$,}
\end{equation}
\begin{equation}      \label{in:nconemultip}
    \lambda_\epsilon\in\ncone{g(x_\epsilon)}{C},
\end{equation}
and
\begin{equation}     \label{in:minconLagr}
    \Lagr{\lambda_\epsilon}{x_\epsilon}\le
   \Lagr{\lambda_\epsilon}{x},\quad\forall
   x\in\ball{x_0}{\epsilon}.
\end{equation}
\end{theorem}

\begin{proof}
Consider mapping $\Image{\mathcal{P}}{x_0}:\X\longrightarrow
\R\times\Y$ associated with problem $(\mathcal{P})$ according to
$(\ref{eq:imagemap})$.
Under the current hypotheses Theorem \ref{thm:Polyakteo}
ensures the existence of $\epsilon_0>0$ such that $\Image{\mathcal{P}}
{x_0}(\ball{x_0}{\epsilon})$ is convex for every $\epsilon\in
(0,\epsilon_0]$. Fix an arbitrary $\epsilon\in (0,\epsilon_0]$ and
define
\begin{equation}    \label{eq:detauinf}
  \tau=\inf\{t:\ (t,y)\in\Image{\mathcal{P}}{x_0}(\ball{x_0}{\epsilon})\cap
   \mathcal{Q}\}.
\end{equation}
Notice that, since $x_0$ can not be a solution to $(\mathcal{P})$
(nor even a local one) because of hypothesis $(iii)$, then
$\Image{\mathcal{P}}{x_0}(\ball{x_0}{\epsilon})\cap\mathcal{Q}
\ne\varnothing$. For the same reason, it is readily seen that
$$
    \tau=\inf\{t:\ (t,y)\in\Image{\mathcal{P}}{x_0}(\ball{x_0}{\epsilon})\cap
   ((-\infty,0]\times C)\}.
$$
According to Lemma \ref{lem:closeimage} set $\Image{\mathcal{P}}{x_0}
(\ball{x_0}{\epsilon})$ can be assumed to be closed. As a closed convex set, by the
Mazur's theorem it is also weakly closed and, by continuity of $f$, bounded.
Therefore, being $(\Y,\|\cdot\|)$ reflexive, $\Image{\mathcal{P}}{x_0}(\ball
{x_0}{\epsilon})$ turns out to be weakly compact. Since $(-\infty,0]
\times C$ is weakly closed as well, again by the Mazur's theorem,
then also $\Image{\mathcal{P}}{x_0}(\ball{x_0}{\epsilon})\cap
((-\infty,0]\times C)$ turns out to be a weakly compact subset
of $\R\times\Y$. The projection function $(t,y)\mapsto t$ is
continuous and convex and thereby it is also lower semicontinuous
with respect to the weak topology, again as a consequence
of the Mazur's theorem. Thus, the infimum in $(\ref{eq:detauinf})$
is actually attained. In other words, there exists $(\hat t,\hat y)
\in\Image{\mathcal{P}}{x_0}(\ball{x_0}{\epsilon})\cap
\mathcal{Q}$, where $\hat t=\tau$. By definition of
$\Image{\mathcal{P}}{x_0}$, this means that there exists $\hat
x\in\ball{x_0}{\epsilon}$ such that
$$
    \hat t=\varphi(\hat x)-\varphi(x_0)<0,\qquad
   \hat y=g(\hat x)\in C.
$$
Therefore it is possible to set $x_\epsilon=\hat x$ to get the first
assertion in the thesis. Indeed, assume ab absurdo the existence
of $\tilde x\in\ball{x_0}{\epsilon}\cap R$ such that $\varphi(\tilde x)<
\varphi(\hat x)$. Then, it follows
$$
    \varphi(\tilde x)-\varphi(x_0)=\varphi(\tilde x)-\varphi(\hat x)
    +\varphi(\hat x)-\varphi(x_0)<\varphi(\hat x)-\varphi(x_0)=
   \tau.
$$
Consequently, it is $( \varphi(\tilde x)-\varphi(x_0),g(\tilde x))
\in\Image{\mathcal{P}}{x_0}(\ball{x_0}{\epsilon})\cap\mathcal{Q}$,
but such an inclusion clearly contradicts the definition of $\tau$.
Observe that, being $(\hat t,\hat y)\in\bd \Image{\mathcal{P}}{x_0}
(\ball{x_0}{\epsilon})$, then according to what noticed in Remark
\ref{rem:Polyakteo} $(iii)$, it is $\hat x\in\bd\ball{x_0}{\epsilon}$.

The second part of the thesis is a straightforward consequence
of the first one. Note that, by optimality of $\hat x$, one has
$$
   \Image{\mathcal{P}}{\hat x}(\ball{x_0}{\epsilon})\cap\mathcal{Q}
  =\varnothing.
$$
Moreover, being
$$
    \Image{\mathcal{P}}{\hat x}(x)=  \Image{\mathcal{P}}{x_0}(x)+
   w_{x_0,\hat x},\quad\forall x\in\X,
$$
where $w_{x_0,\hat x}=(\varphi(x_0)-\varphi(\hat x),\nullv)$, then
$\Image{\mathcal{P}}{\hat x}(\ball{x_0}{\epsilon})$ is a mere translation
of $\Image{\mathcal{P}}{x_0}(\ball{x_0}{\epsilon})$. Thus, set
$\Image{\mathcal{P}}{\hat x}(\ball{x_0}{\epsilon})$ is a convex
subset of $\R\times\Y$ with nonempty interior (recall Remark
\ref{rem:Polyakteo} $(iii)$) and disjoint from $\mathcal{Q}$.
According to the Eidelheit's theorem (see, for example,
\cite{Zali02}), it is
then possible to linearly separate $\Image{\mathcal{P}}{\hat x}
(\ball{x_0}{\epsilon})$ and $\cl\mathcal{Q}$, what means that
there exist $(\rho_\epsilon,\lambda_\epsilon)\in (\R\times\Y^*)
\backslash\{(0,\nullv^*)\}$ and $\alpha\in\R$ such that
\begin{equation}     \label{in:separ1}
    \rho_\epsilon(\varphi(x)-\varphi(\hat x))+\langle\lambda_\epsilon
   ,g(x)\rangle\ge\alpha,\quad\forall x\in\ball{x_0}{\epsilon},
\end{equation}
and
\begin{equation}      \label{in:separ2}
    \rho_\epsilon t+\langle \lambda_\epsilon,y\rangle\le\alpha,
   \quad\forall (t,y)\in \cl\mathcal{Q}=(-\infty,0]\times C.
\end{equation}
If taking $x=\hat x$ in inequality $(\ref{in:separ1})$, one finds
$$
    \langle\lambda_\epsilon,g(\hat x)\rangle\ge\alpha.
$$
On the other hand, being $(0,g(\hat x))\in\cl\mathcal{Q}$, from
inequality $(\ref{in:separ2})$ one gets
$$
    \langle\lambda_\epsilon,g(\hat x)\rangle\le\alpha,
$$
wherefrom one deduces
\begin{equation}    \label{eq:nonrigcomp}
   \langle\lambda_\epsilon,g(\hat x)\rangle=\alpha.
\end{equation}
If taking now an arbitrary $y\in C$, then being $(0,y)\in\cl
\mathcal{Q}$, by inequality $(\ref{in:separ2})$ it results in
$$
     \langle\lambda_\epsilon,y\rangle\le\alpha,
$$
and hence
$$
     \langle\lambda_\epsilon,y-g(\hat x)\rangle\le 0,\quad\forall
    y\in C.
$$
This shows that $\lambda_\epsilon\in\ncone{g(\hat x)}{C}$. Again, since
$(-1,g(\hat x))\in\cl\mathcal{Q}$, from inequality $(\ref{in:separ2})$
it follows that $\rho_\epsilon\ge 0$. Let us show now that, under
the current hypotheses, actually it is $\rho_\epsilon>0$, so up to a
rescaling of $\lambda_\epsilon$ it is possible to take $\rho_\epsilon=1$.
Indeed, assume to the contrary that $\rho_\epsilon=0$. Since by
virtue of hypothesis $(iii)$ mapping $g$ is metrically regular
around $(x_0,g(x_0))$, one has
$$
    g(\ball{x_0}{r})\supseteq\ball{g(x_0)}{\sigma r}
$$
for positive $\sigma$ and $r<\epsilon$. From inequality $(\ref{in:separ1})$
it follows
$$
   \langle\lambda_\epsilon,g(x_0)+\eta u\rangle\ge\alpha,
   \quad\forall u\in\Sfer,
$$
with $0<\eta<\sigma r$. Being $(0,g(x_0))\in\cl\mathcal{Q}$,
owing to $(\ref{in:separ2})$ one has
$$
   \langle\lambda_\epsilon,g(x_0)\rangle\le\alpha.
$$
Thus, one finds
$$
   \eta\langle\lambda_\epsilon,u\rangle \ge\alpha-
   \langle\lambda_\epsilon,g(x_0)\rangle\ge 0,\quad\forall u\in\Sfer,
$$
which can not be consistent with the fact that $\lambda_\epsilon
\ne\nullv^*$ (remember that $(\rho_\epsilon,\lambda_\epsilon)\in
(\R\times\Y^*)\backslash\{(0,\nullv^*)\}$).

Finally, by using equality
$(\ref{eq:nonrigcomp})$ in $(\ref{in:separ1})$, one obtains
$$
    \Lagr{\lambda_\epsilon}{x}=\varphi(x)+
   \langle\lambda_\epsilon,g(x)\rangle\ge
   \varphi(\hat x)+\langle\lambda_\epsilon,g(\hat x)\rangle=
   \Lagr{\lambda_\epsilon}{\hat x},\quad\forall x\in
   \ball{x_0}{\epsilon}.
$$
This completes the proof.
\end{proof}

\begin{remark}
(i) In a finite dimensional setting the existence of a solution
to $(\mathcal{P}_{x_0,\epsilon})$ is automatic, as an obvious
consequence of the Weierstrass theorem. In that case, indeed,
set $\ball{x_0}{\epsilon}\cap R$ is compact, $g$ being continuous
near $x_0$. If $\X$ is infinite dimensional the solution existence
becomes a by-product of the convexity hidden in the problem
localization. Observe that, since $\ball{x_0}{\epsilon}\cap R$ is
not necessarily convex, it may fail to be weakly closed. Analogously,
since $\varphi$ is not convex, nothing can be said about its weak
lower semicontinuity. Therefore, arguments based on weak compactness
in a reflexive space can not be invoked directly, without passing
through the convexity principle.

(ii) The optimality condition expressed by $(\ref{in:nconemultip})$
and $(\ref{in:minconLagr})$ can be regarded as another manifestation
of the convexity behaviour of $(\mathcal{P}_{x_0,\epsilon})$. The
property for a solution to be minimal also for the Lagrangian
function $\Lagr{\lambda_\epsilon}{\cdot}$, while it is typical in
convex optimization, is a circumstance generally failing in nonlinear
programming.

(iii) A feature of Theorem \ref{pro:Lagroptcond} to be underlined
is that such result guarantees the existence of regular Lagrange
multipliers, i.e. multipliers with nonnull first component $\rho_\epsilon$.
\end{remark}


\subsection{Lagrangian duality}

In its general form, a Lagrangian duality scheme can be defined
whenever the following elements are given: a function $\mathscr{L}:
A\times B\longrightarrow\R\cup\{\pm\infty\}$, where $A$ and $B$ are
arbitrary sets, and subsets $S_A\subseteq A$ and $S_B\subseteq B$.
In considering the extremum problems
$$
    \min_{b\in S_B}\sup_{a\in S_A}\NlLagr{a}{b}   
   \leqno (\mathcal{P}_\mathscr{L}) 
$$
and
$$
    \max_{a\in S_A} \inf_{b\in S_B}\NlLagr{a}{b},   
   \leqno (\mathcal{P}_\mathscr{L}^*) 
$$
a Lagrangian duality scheme singles out two fundamental concepts:
one is the {\em duality gap}, i.e. the difference of the respective optimal values
of the problems $(\mathcal{P}_\mathscr{L})$ and $(\mathcal{P}_\mathscr{L}^*)$
$$
    \min_{b\in S_B}\sup_{a\in S_A}\NlLagr{a}{b}-
    \max_{a\in S_A} \inf_{b\in S_B}\NlLagr{a}{b},
$$
provided that such difference is defined (it is, if  $(\mathcal{P}_\mathscr{L})$ and
$(\mathcal{P}_\mathscr{L}^*)$ do not happen to have the same infinite optimal value).
The other one is a {\em saddle point} for $\mathscr{L}$, i.e. any
element $(\bar a,\bar b)\in\ S_A\times S_B$ satisying the inequalities
$$
    \NlLagr{a}{\bar b}\le\NlLagr{\bar a}{\bar b}\le\NlLagr{\bar a}{b},
    \quad\forall (a,b)\in S_A\times S_B.
$$
In this context, any function such as $\mathscr{L}$
is usually called the Lagrangian function
associated with the duality scheme. In such a general setting, the
following well-known proposition explains the role of the aforementioned
concepts (for its proof, which is elementary, see for instance
\cite{BonSha00})

\begin{proposition}     \label{pro:dualgap}
Whenever it is defined, the duality gap is nonnegative, that is
$$
    \sup_{a\in S_A} \inf_{b\in S_B}\NlLagr{a}{b}\le
    \inf_{b\in S_B}\sup_{a\in S_A}\NlLagr{a}{b}.
$$
Moreover, the function $\mathscr{L}$ admits a saddle point iff
problems $(\mathcal{P}_\mathscr{L})$ and $(\mathcal{P}_\mathscr{L}^*)$
share the same optimal value and each has nonempty set of optimal
solutions. In that case the set of saddle points for $\mathscr{L}$
coincides with the Cartesian product of the respective optimal
solution sets.
\end{proposition}

In view of the above result, it becomes crucial to find out
verifiable conditions on problem data, under which a saddle point
exists.

Now, as a consequence of Theorem \ref{pro:Lagroptcond},
it turns out that a Lagrangian duality scheme, through a localization
of problem $(\mathcal{P})$, can be performed by making use of the
simplest type of Lagrangian function, namely the linear one,
provided that $C$ is a cone. Thus, in such event, the extremum
problems in duality are
$$
   \min_{x\in\ball{x_0}{\epsilon}}\sup_{y^*\in C^\ominus}
    \Lagr{y^*}{x} \leqno (\mathcal{P}_{\rm L}) 
$$
and
$$
   \max_{y^*\in C^\ominus}\inf_{x\in\ball{x_0}{\epsilon}}
    \Lagr{y^*}{x}.     \leqno (\mathcal{P}_{\rm L}^*)
$$

\begin{theorem}
With reference to problem $(\mathcal{P})$, suppose that $C$ is
a nonempty closed convex cone and $x_0\in R$.
Under the hypotheses of Theorem \ref{pro:Lagroptcond},
there exists $\epsilon_0>0$ such that for every
$\epsilon\in (0,\epsilon_0]$ there are $(x_\epsilon,\lambda_\epsilon)\in
\bd\ball{x_0}{\epsilon}\times (C^\ominus\cap g(x_\epsilon)^\perp)$
such that $(x_\epsilon,\lambda_\epsilon)$ is a saddle point for ${\rm L}$.
Consequently, the related duality gap is $0$ and both the primal and
the dual problem have nonempty solution sets.
\end{theorem}

\begin{proof}
It is readily seen that if $C$ is a cone and $g(x_\epsilon)\in C$,
the inclusion $\lambda_\epsilon\in\ncone{g(x_\epsilon)}{C}$ implies
$\lambda_\epsilon\in C^\ominus\cap g(x_\epsilon)^\perp$. Indeed,
if taking $y=2g(x_\epsilon)$ and $y=\nullv$ in the inequality
$$
    \langle \lambda_\epsilon,y-g(x_\epsilon)\rangle\le 0,
$$
one obtains two inequalities, which can be consistent only if
$\lambda_\epsilon\in g(x_\epsilon)^\perp$. Taking this fact into
account, the last inequality gives also $\lambda_\epsilon\in
C^\ominus$. Take $\epsilon\in (0,\epsilon_0]$, where $\epsilon_0$
is as in Theorem \ref{pro:Lagroptcond}. By applying inequality
$(\ref{in:minconLagr})$, one finds
$$
   \Lagr{\lambda}{x_\epsilon}=\varphi(x_\epsilon)+\langle\lambda,
  g(x_\epsilon)\rangle\le\varphi(x_\epsilon)=
  \Lagr{\lambda_\epsilon}{x_\epsilon}\le\Lagr{\lambda_\epsilon}{x},
   \quad\forall (\lambda,x)\in C^\ominus\times\ball{x_0}{\epsilon}.
$$
The last assertion in the thesis immediately follows from Proposition
\ref{pro:dualgap}.
\end{proof}


\subsection{Problem calmness}

This subsection focuses on some properties of constrained
extremum problems in the presence of perturbations. The perturbation
analysis of optimization problems has revealed to be able to afford
useful theoretical insights into the very nature of the issue.
The format of parametric problems here in consideration is as follows
$$
    \min_{x\in\X} \varphi(x)\quad\hbox{subject to}\quad g(x)+y\in C,
    \leqno (\mathcal{P}_y)
$$
where $y\in\Y$ plays the role of a parameter. The corresponding
feasible region is given therefore by $R(y)=g^{-1}(C-y)$. 
A notion capturing a sensibility behaviour with respect to perturbations
near a reference value is that of problem calmness. Proposed by
R.T. Rockafellar, such notion appeared firstly in \cite{Clar76}
and since then it was largely employed in perturbation analysis
of optimization problems and related fields.

\begin{definition}
With reference to a class of problems $(\mathcal{P}_y)$, let
$\hat x\in R(\nullv)$ be a solution to $(\mathcal{P}_\nullv)$.
Problem $(\mathcal{P}_\nullv)$ is said to be {\em calm} at $\hat x$ if
there exists a constant $r>0$ such that
$$
   \inf_{y\in r\B\backslash\{\nullv\}}\inf_{x\in R(y)\cap
  \ball{\hat x}{r}}\frac{\varphi(x)-\varphi(\hat x)}{\|y\|}
   >-\infty.
$$
\end{definition}

Following a successful approach to this topic, sufficient conditions for
problem calmness can be achieved by studying the localized (optimal)
value function associated with $(\mathcal{P}_y)$, i.e. function
$\val{x_0,\epsilon}:\Y\longrightarrow\R\cup\{\pm\infty\}$ defined by
$$
    \val{x_0,\epsilon}(y)=\inf_{x\in R(y)\cap\ball{x_0}{\epsilon}}
   \varphi(x).
$$
In particular, the property of $\val{x_0,\epsilon}$ to be calm from
below at $\nullv$ appeared to be adequate to this aim. Recall that
a function $\phi:\Y\longrightarrow\R\cup\{\pm\infty\}$ is said to be
{\em calm from below} at $y_0$ if $y_0\in\dom\phi$ and it holds
$$
     \liminf_{y\to y_0}\frac{\phi(y)-\phi(y_0)}{\|y-y_0\|}>-\infty.
$$
In turn, calmness from below for function can be easily
obtained from the subdifferentiability property.

\begin{theorem}       \label{them:subdifval}
Given a class of perturbed problem $(\mathcal{P}_y)$, let $x_0
\in R(\nullv)\cap\Omega$, where $\Omega$ is an open subset
of $\X$. Suppose that:

\noindent $(i)$ $(\X,\|\cdot\|)$ is uniformly convex and its
convexity modulus fulfils condition $(\ref{in:modcocond})$;

\noindent $(ii)$ $(\Y,\|\cdot\|)$ is a reflexive Banach space; 

\noindent $(iii)$ $\varphi,\, g\in\CII(\Omega)$ and $\Fder(\varphi,g)
(x_0)\in\Lbo(\X,\R\times\Y)$ is onto.

\noindent Then, there exists a positive $\epsilon_0$ such that 
for every $\epsilon\in (0,\epsilon_0]$ it holds
$$
   \partial \val{x_0,\epsilon}(\nullv)\ne\varnothing.
$$
Consequently, function $\val{x_0,\epsilon}$ is calm from below at
$\nullv$.
\end{theorem}

\begin{proof}
The subdifferentiability of $\val{x_0,\epsilon}$ at $\nullv$ can be
achieved as a further consequence of the possibility of separating
$\Image{\mathcal{P}}{x_0}(\ball{x_0}{\epsilon})$ and $\cl\mathcal{Q}$
for every $\epsilon\in (0,\epsilon_0]$, where $\epsilon_0$ is a
positive constant as in Theorem \ref{pro:Lagroptcond}. Indeed,
fix an arbitrary $y\in\Y$. By applying Theorem \ref{pro:Lagroptcond}
to $(\mathcal{P}_\nullv)$, one gets $x_\epsilon\in\bd\ball{x_0}
{\epsilon}$ and $\lambda_\epsilon\in\Y^*$ satisfying $(\ref{sol:optim})$,
$(\ref{in:nconemultip})$ and $(\ref{in:minconLagr})$. It follows
\begin{equation}      \label{in:lagrmin}
   \langle\lambda_\epsilon,g(x_\epsilon)\rangle+\varphi(x_\epsilon)
   \le\varphi(x)+\langle\lambda_\epsilon,g(x)\rangle,\quad\forall
   x\in\ball{x_0}{\epsilon}.
\end{equation}
Because of $(\ref{in:nconemultip})$, whenever $x\in R(y)\cap
\ball{x_0}{\epsilon}$, being $g(x)+y\in C$ one has
$$
    \langle\lambda_\epsilon,g(x)-g(x_\epsilon)\rangle\le
   -\langle\lambda_\epsilon,y\rangle.
$$
The last inequality on account of $(\ref{in:lagrmin})$ gives
$$
    0\le\varphi(x)-\varphi(x_\epsilon)+\langle\lambda_\epsilon,
    g(x)-g(x_\epsilon)\rangle\le
    \varphi(x)-\varphi(x_\epsilon)-\langle\lambda_\epsilon,y\rangle,
    \quad\forall x\in R(y)\cap\ball{x_0}{\epsilon},
$$
whence 
$$
   \langle\lambda_\epsilon,y\rangle\le\inf_{x\in R(y)\cap\ball{x_0}
  {\epsilon}}\varphi(x)-\varphi(x_\epsilon)=
   \val{x_0,\epsilon}(y)-\val{x_0,\epsilon}(\nullv).
$$
By arbitrariness of $y\in\Y$ the first assertion in the thesis
is proved. The second one is a straightforward consequence
of the first one. Indeed, obviously $\nullv\in\dom\val{x_0,\epsilon}$
and it holds
$$
   \liminf_{y\to\nullv}\frac{\val{x_0,\epsilon}(y)-
   \val{x_0,\epsilon}(\nullv)}{\|y\|}\ge\inf_{u\in\Sfer}
   \langle\lambda_\epsilon,u\rangle\ge
   -\|\lambda_\epsilon\|>-\infty.
$$
This completes the proof.
\end{proof}

\begin{corollary}
Under the hypotheses of Theorem \ref{them:subdifval}, there exists
$\epsilon_0>0$ such that, for every $\epsilon\in (0,\epsilon_0]$,
$(\mathcal{P}_\nullv)$ admits a corresponding $\epsilon$-localization,
which is calm at a respective solution $x_\epsilon\in\ball{x_0}
{\epsilon}$.
\end{corollary}

\begin{proof}
According to Theorem \ref{pro:Lagroptcond}, an $\epsilon_0>0$
exsists such that, for every $\epsilon\in (0,\epsilon_0]$, each
$\epsilon$--localization of $(\mathcal{P}_\nullv)$ admits a solution
$x_\epsilon\in\bd\ball{x_0}{\epsilon}$. Thus, fixed $r>0$,
using the calmness from below of function $\val{x_0,\epsilon}$ at $\nullv$, as it holds
by definition
$$
    \varphi(x)\ge\val{x_0,\epsilon}(y),\quad\forall x\in
   R(y)\cap\ball{x_0}{\epsilon},
$$
one obtains
$$
    \inf_{y\in r\B\backslash\{\nullv\}}
    \inf_{x\in R(y)\cap\ball{x_0}{\epsilon}\cap\ball{x_\epsilon}{r}}
    \frac{\varphi(x)-\varphi(x_\epsilon)}{\|y\|}\ge
    \inf_{y\in r\B\backslash\{\nullv\}}
    \frac{\val{x_0,\epsilon}(y)-\val{x_0,\epsilon}(\nullv)}{\|y\|}
    >-\infty.
$$
The proof is complete.
\end{proof}

\vskip1cm

\bibliographystyle{amsplain}

\begin{thebibliography}{99}


\bibitem{Barv95} {\sc Barvinok, A.I.}, \textit{Problems of distance
geometry and convex properties of quadratic maps}, Discrete Comput.
Geom. \textbf{13} (1995), no. 2, 189--202.

\bibitem{BoEmKo04} {\sc Bobylev, N.A.}, {\sc Emelyanov, S.V.},
and {\sc Korovin, S.K.}, \textit{Convexity of images of convex
sets under smooth maps}, Comput. Math. Model. \textbf{15} (2004),
no. 3, 213--222.

\bibitem{BonSha00} {\sc Bonnans, J.F.} and {\sc Shapiro, A.},
\textit{Perturbation analysis of optimization problems}, Springer-Verlag,
New York, 2000.

\bibitem{Clar76} {\sc Clarke, F.H.}, \textit{A new approach to Lagrange
multipliers}, Math. Oper. Res. \textbf{1} (1976), no. 2, 165--174.

\bibitem{Dies85} {\sc Diestel, J.}, \textit{Geometry of Banach
spaces--selected topics}, Lecture Notes in Mathematics, Vol. 485,
Springer-Verlag, Berlin-New York, 1975.

\bibitem{DmMiOs80} {\sc Dmitruk, A.V., Milyutin, A.A.} and {\sc Osmolovski\u i, N.P.},
\textit{Lyusternik's theorem and the theory of extrema},
Russian Math. Surveys \textbf{35}(6) (1980), 11--51.

\bibitem{FaHaMoPeZi01} {\sc Fabian, M.}, {\sc Habala, P.}, {\sc H\'ajek, P.},
{\sc Montesinos Santaluc\'ia, V.}, {\sc Pelant, J.}, and {\sc Zizler, V.},
\textit{Functional analysis and infinite-dimensional geometry},
Springer-Verlag, New York, 2001.

\bibitem{HirPla89} {\sc Hiriart-Urruty, J.-B.} and {\sc Plazanet, Ph.},
\textit{Moreau's decomposition theorem revisited}, Ann. Inst. H.
Poincaré Anal. Non Lin\'eaire \textbf{6} (1989), 325--338.

\bibitem{KanAki82}  {\sc Kantorovich, L.V.} and {\sc Akilov, G.P.},
\textit{Functional analysis}, Pergamon Press, Oxford-Elmsford,
N.Y., 1982.

\bibitem{Milm71}  {\sc Milman, V.D.}, \textit{Geometric theory of
Banach spaces. II. Geometry of the unit ball}, Uspekhi Mat. Nauk
\textbf{26} (1971), no. 6, 73--149 [in Russian].

\bibitem{Mord06} {\sc Mordukhovich, B.S.}, \textit{Variational Analysis
and Generalized Differentiation I: Basic Theory}, Springer-Verlag,
Berlin Heidelberg, 2006.

\bibitem{Phel93} {\sc Phelps, R.R.}, \textit{Convex functions, monotone
operators and differentiability}, Second edition. Lecture Notes in
Mathematics, 1364. Springer-Verlag, Berlin, 1993.

\bibitem{Poly98} {\sc Polyak, B.T.}, \textit{Convexity of quadratic
transformations and its use in control and optimization}, J. Optim.
Theory Appl. \textbf{99} (1998), no. 3, 553--583.

\bibitem{Poly01} {\sc Polyak, B.T.}, \textit{Convexity of Nonlinear
Image of a Small Ball with Applications to Optimization}, Set-Valued
Anal. \textbf{9} (2001), no. 1-2, 159--168.

\bibitem{Poly03} {\sc Polyak, B.T.}, \textit{The convexity principle
and its applications}, Bull. Braz. Math. Soc. (N.S.) \textbf{34}
(2003), no. 1, 59--75.

\bibitem{Zali02} {\sc Z\u alinescu, C.}, \textit{Convex analysis
in general vector spaces}, World Scientific Publishing Co., Inc.,
River Edge, NJ, 2002.

%

\end{thebibliography}

\vskip1cm

\end{document}